\documentclass[10pt,reqno]{amsart}

\usepackage{amsmath,amssymb}
\usepackage{graphicx}
\usepackage{algorithm}
\usepackage{enumerate}
\usepackage{tikz}
\usetikzlibrary{decorations.pathmorphing}
\usetikzlibrary{decorations.markings}
\usetikzlibrary{positioning}
\usetikzlibrary{shapes}

\usepackage{tabu}
\usepackage{url}

\usepackage{mathrsfs}
\usepackage{caption} 
\usepackage{subcaption}

\usetikzlibrary{matrix, calc, arrows}
\usetikzlibrary{graphs}
\usetikzlibrary{graphs.standard}
\usepackage{tkz-berge}

\usepackage{todonotes}

\newtheorem{theorem}{Theorem}[section]
\newtheorem{lemma}[theorem]{Lemma}
\newtheorem{corollary}[theorem]{Corollary}
\newtheorem{proposition}[theorem]{Proposition}

\theoremstyle{definition}
\newtheorem{example}[theorem]{Example}
\newtheorem{definition}[theorem]{Definition}

\newcommand{\rmv}[1]{}

\title{Properties of Partial Dominating Sets of Graphs}
\author[Case, Fenstermacher, Ganguly, Laskar]{Benjamin M. Case, Todd Fenstermacher, Soumendra Ganguly, Renu C. Laskar}
\address{School of Mathematical and Statistical Sciences, Clemson University, SC, USA}
\thanks{Benjamin M. Case was partially supported by the  National Science Foundation  under grants DMS-1403062 and DMS-1547399.}

\begin{document}

	\begin{abstract}
		A set $S\subseteq V$ is a \textit{dominating set} of $G$ if every vertex in $V - S$ is adjacent to at least one vertex in $S$. The \textit{domination number} $\gamma(G)$ of $G$ equals the minimum cardinality of a dominating set $S$ in $G$; we say that such a set $S$ is a $\gamma$\textit{-set}. A generalization of this is partial domination which was introduced in 2017 by Case, Hedetniemi, Laskar, and Lipman \cite{CHLL17arXiv,CHLL17} .  In \emph{partial domination} a set $S$ is a \emph{$p$-dominating set} if it dominates a proportion $p$ of the vertices in $V$. The \emph{p-domination number} $\gamma_{p}(G)$ is the minimum cardinality of a $p$-dominating set in $G$.  In this paper, we investigate further properties of partial dominating sets, particularly ones related to graph products and locating partial dominating sets.  We also introduce the concept of a \emph{$p$-influencing set} as the union of all $p$-dominating sets for a fixed $p$ and investigate some of its properties. 
		\medskip
		
		
		%
		%
	\end{abstract}
	\maketitle
	\section{Introduction}
	
	Let $G=(V,E)$ be a graph with vertex set $V=\{v_1,v_2,...,v_n\}$ and \emph{order} $n = |V|$. The {\em open neighborhood} of a vertex $v$ is the set $N(v) := \{u\: |\: uv \in E\} $ of vertices $u$ that are adjacent to $v$; the \emph{closed neighborhood} of $v$ is $N[v]:=N(v)\cup \{v\}.$ A set $S\subseteq V$ is a \emph{dominating set} of $G$ if every vertex in $V - S$ is adjacent to at least one vertex in $S$, or equivalently, if $N[S] := \bigcup_{u\in S} N[u] = V$. The \emph{domination number} $\gamma (G)$ of $G$ equals the minimum cardinality of a dominating set $S$ in $G$; we say that such a set $S$ is a \emph{$\gamma$-set}. Domination has been a well studied area for many years \cite{Bollobas,Cock1978,Goncalves,Harary,Hutson}.
	
	For any graph $G=(V,E)$ and proportion $p \in [0,1]$, a set $S\subseteq V$ is a \textit{p-dominating set} if \[\frac{|N[S]|}{|V|} \geq p.\]  The \textit{p-domination number} $\gamma_p(G)$ equals the minimum cardinality of a $p$-dominating set in $G$.  Partial domination was first introduced by  Case, Hedetniemi, Laskar, and Lipman \cite{CHLL17arXiv,CHLL17} in 2017.  Around the same time, in an independent work of Das the same concept was introduced \cite{Das17arXiv,Das2018}.

	As noted in \cite{CHLL17arXiv}, a $\gamma_p$-set is not in general related to a $\gamma$-set. In particular, a $\gamma$-set does not necessarily contain a $\gamma_{p}$-set. Equivalently, a $\gamma_{p}$-set cannot necessarily be extended to a $\gamma$-set. To see this, it is helpful to revisit the subdivided star graph in Figure \ref{ex:Disjoint}, where the $\gamma$-set denoted by triangles is disjoint from $\gamma_{1/2}$-set consisting of just the square vertex. 
	
	\begin{figure}[hh] 
		\centering
		\begin{tikzpicture}
		\node [draw,rectangle] (A) at (0,0) {};
		\node [draw,regular polygon, regular polygon sides=3,inner sep=1.5pt] (B) at (1,0) {};
		\node [draw,regular polygon, regular polygon sides=3,inner sep=1.5pt] (C) at (.707,.707) {};
		\node [draw,regular polygon, regular polygon sides=3,inner sep=1.5pt] (D) at (0,1) {};
		\node [draw,regular polygon, regular polygon sides=3,inner sep=1.5pt] (E) at (-.707,.707) {};
		\node [draw,regular polygon, regular polygon sides=3,inner sep=1.5pt] (F) at (-1,0) {};
		\node [draw,regular polygon, regular polygon sides=3,inner sep=1.5pt] (G) at (-.707,-.707) {};
		\node [draw,regular polygon, regular polygon sides=3,inner sep=1.5pt] (H) at (0,-1) {};
		\node [draw,regular polygon, regular polygon sides=3,inner sep=1.5pt] (I) at (.707,-.707) {};
		\node [draw,circle] (J) at (2,0) {};
		\node [draw,circle] (K) at (1.414,1.414) {};
		\node [draw,circle] (L) at (0,2) {};
		\node [draw,circle] (M) at (-1.414,1.414) {};
		\node [draw,circle] (N) at (-2,0) {};
		\node [draw,circle] (O) at (-1.414,-1.414) {};
		\node [draw,circle] (P) at (0,-2) {};
		\node [draw,circle] (Q) at (1.414,-1.414) {};
		
		\draw[line width=1pt] (A) edge (B);
		\draw[line width=1pt] (B) edge (J);
		\draw[line width=1pt] (A) edge (C);
		\draw[line width=1pt] (C) edge (K);
		\draw[line width=1pt] (A) edge (D);
		\draw[line width=1pt] (D) edge (L);
		\draw[line width=1pt] (A) edge (E);
		\draw[line width=1pt] (E) edge (M);
		\draw[line width=1pt] (A) edge (F);
		\draw[line width=1pt] (F) edge (N);
		\draw[line width=1pt] (A) edge (G);
		\draw[line width=1pt] (G) edge (O);
		\draw[line width=1pt] (A) edge (H);
		\draw[line width=1pt] (H) edge (P);
		\draw[line width=1pt] (A) edge (I);
		\draw[line width=1pt] (I) edge (Q);
		\end{tikzpicture}
		\caption{The $\gamma$-set denoted by triangles is disjoint from the $\gamma_{1/2}$-set consisting of just the square vertex. \label{ex:Disjoint}}
	\end{figure}
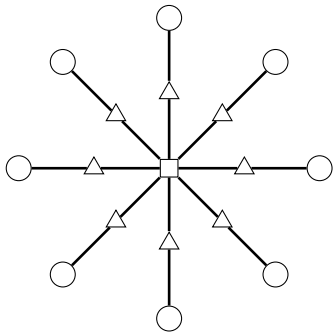

	{\bf Organization of the paper.}  In section \ref{Sec:graphproducts} we generalize Vizing's conjecture about domination in graph products to the setting of partial domination and prove some special cases.  In Section \ref{Sec:locating} we investigate some results related to finding a partial dominating set in a graph. In Section \ref{Sec:pinfluencing} we introduce $p$-influencing sets and consider some properties and examples. 
	
	\section{Graph Products}\label{Sec:graphproducts}
	
	We investigate properties of partial dominating sets related to graph products. 
	In 1986, Vadim G. Vizing conjectured that for domination
	\[	\gamma(G\square H) \geq \gamma(G)\gamma(H).		\]
	Here we conjecture that 
	\[ \gamma_p(G \square H) \geq \gamma_p(G)\gamma_p(H).	\]
	We will primarily be interested in the case when $p = 1/2.$

	\begin{proposition} For paths $P_m$ and $P_n,$ with $n \geq m \geq 2$
		\[ \gamma_{1/2}(P_m \square P_n) \geq \gamma_{1/2}(P_m)\gamma_{1/2}(P_n). \]
	\end{proposition}
	
	\begin{proof} The claim follows directly from Theorem 2.8 in \cite{CHLL17}, which says
		\begin{enumerate}
			\item $\gamma_{1/2}(P_n) = \lceil n/6 \rceil,$
			\item for $m=2,$ $\gamma_{1/2}(P_2 \square P_n) = \lceil n/4 \rceil,$
			\item for $m \geq 3,$ $\gamma_{1/2}(P_m \square P_n) = \lceil mn/10 \rceil.$
		\end{enumerate}
	\end{proof}
	
	Note that if $\gamma_{p}(G) = \gamma_{p}(H) = 1,$ then the conjecture holds trivially (e.g., if $G$ and $H$ are both complete graphs). 
	
	\begin{proposition} For complete graphs, $K_m$ and $K_n,$ 
		\[	\gamma_{1/2}(K_m \square K_n) = \left\lceil \frac{m+n - \sqrt{m^2 + n^2}}{2} \right\rceil. \]
	\end{proposition}
	
	\begin{proof} We consider choosing vertices for a $\gamma_{1/2}$-set $S.$ Say the first vertex added to $S$ is $v_1=(v,v').$ Note that $v$ dominates exactly $m+n-1$ vertices. If $m+n-1 \geq mn/2,$ we are done. If not, we must add another vertex $u$ to $S.$ Note we essentially have three options here: 
		\[(1) \hspace{1ex} v_2 = (v,u') \hspace{1cm} (2) \hspace{1ex} v_2 = (u,v') \hspace{1cm}    (3) \hspace{1ex} v_2 = (u,u')  \]
		where $u \neq v$ and $u' \neq v'.$ In the first case, $v_2$ dominates exactly $m-1$ vertices not already dominated by $S.$ In the second case $v_2$ dominates $n-1$ such vertices, and in the last case $v_2$ dominates $m+n-3$ such vertices. 
		
		Now in general, when adding the next vertex to $S,$ it is helpful to consider which copies of $K_m$ and $K_n$ the vertex is coming from relative to the last vertex chosen. We have these three options: (1) choose the same $K_m$ but a different $K_n,$ (2) choose a different $K_m$ but the same $K_n,$ (3) choose a different $K_m$ and $K_n.$ In general option (3) is optimal, followed by option (1), then option (2). When option (3) is used every time in such a way that each vertex added to $S$ comes from a copy of $K_m$ and $K_n$ not previously used, then the $k$th vertex dominates $m+n-2k+1$ vertices not already dominated by $S.$ We claim this is the optimal strategy in choosing vertices for a $\gamma_{1/2}$-set. We refer to it as $*$
		
		To see this, consider alternative choice of vertices, where we execute option (2) optimally $\ell$ times, and then execute option (3) optimally until we have a $1/2$-dominating set. We refer to this as $**.$ Now compare the number of vertices dominated by $*$ and $**$ at each step. There is no difference for the first vertex. For the second, $**$ is dominates $n-2$ fewer than $*.$ In general, for the $k$th vertex added in $**$ with $2 \leq k \leq\ell+1,$ $**$ dominates $(k-1)(n-k)$ fewer vertices than $*$. So at step $\ell+1,$ $**$ dominates $\ell(n-\ell-1)$ fewer vertices than $*.$ Now for each vertex added after step $\ell+1,$ $**$ dominates $\ell$ more vertices than $*.$ Thus, it will take $**$ 
		$(n-\ell-1)+1+\ell = n$ steps to dominate the same number of vertices as $*.$ However, $*$ will have produced a $1/2$-dominating set well before the $n$th step. 
		Using a similar argument, we can show that any choice of vertices different than $*$ will be suboptimal. 
		
		Now the cardinality of the $\gamma_{1/2}$-set found by $*$ is 
		\[	\min\{	k \text{ }|\text{ } \sum_{i=1}^k \left( m + n - (2i-1) \right) \geq \frac{mn}{2}	\}.	\]
		
		But note that $\sum_{i=1}^k \left( m + n - (2i-1) \right) = k(m+n) - k^2.$
		So the inequality to solve is 
		\[	k(m+n) - k^2 \geq \frac{mn}{2}		\]
		which has minimum positive integer solution 
		\[ k = 	\left\lceil \frac{m+n - \sqrt{m^2 + n^2}}{2} \right\rceil	.	\]
	\end{proof}

	\begin{proposition}
		For a path $P_n$ and complete graph $K_m,$ 
		\[ \gamma_{1/2}(P_n \square K_m) = \left \lceil \frac{mn}{2(m+2)} \right\rceil. \]
	\end{proposition}
	
	\begin{proof} The maximum degree of a vertex in $P_n \square K_m$ is $m+2.$ Consider the set of vertices $S$ of maximum cardinality such that each vertex in $S$ dominates exactly $m+2$ vertices, and no two vertices in $S$ dominate the same vertex. The cardinality of $S$ is  $\left\lceil \frac{n-2}{2} \right\rceil.$ 
		Now write $n=2k$ $(n=2k-1)$ for even (odd) $n$ and $k\in \mathbb{N}.$  Note that in either case 
		\[ \left\lceil \frac{n-2}{2} \right\rceil = k-1. \]
		That is $S$ has cardinality $k-1.$ 
		
		Now suppose $(m+2)(k-1) \geq mn/2.$ Then either $S$ or a subset of $S$ is a $\gamma_{1/2}$-set, and \[ \gamma_{1/2}(P_n \square K_m) = \left \lceil \frac{mn}{2(m+2)} \right\rceil. \]
		
		On the other hand, suppose $(m+2)(k-1) < mn/2.$ Then neither $S$ nor any subset of $S$ is a $\gamma_{1/2}$-set. However, we may add a vertex $v$ to $S$ so $S\cup\{v\}$ dominates at least $k$ copies of $K_m.$ That is, $S \cup \{v\}$ is a $\gamma_{1/2}$-set. Moreover, we have 
		\[  k(m+2) > \frac{mn}{2} > (k-1)(m+2)  \implies k = \left \lceil \frac{mn}{2(m+2)} \right\rceil. \]
		
	\end{proof}

	
	
	\begin{proposition} Let $G$ be a graph of order $n,$
		\[ \gamma_{1/2}(G \square P_2) \geq \gamma_{1/2}(G). \]
	\end{proposition}
	
	\begin{proof}
		Let $S$ be a $\gamma_{1/2}(G \square P_2)$-set. Now think of $G \square P_2$ as two copies of $G,$ say $G_1$ and $G_2.$ Then since $|N[S]| \geq n,$ we have WLOG that $|N[S] \cap G_1| \geq n/2.$ Denote $S\cap G_i$ as $S_i.$ Now consider the vertex set \[S' = S_1 \cup \{v \hspace{1ex}|\hspace{1ex}  v \in N[S_2]\cap G_1	\}.\] Then $S'$ dominates at least half of $G,$ and $|S'| \leq |S|.$ 
		
		Thus, given a $\gamma_{1/2}(G \square P_2)$-set, we can find a vertex set of $G$ of size at most $\gamma_{1/2}(G \square P_2)$ which dominates $1/2$ of $G.$ So $\gamma_{1/2}(G \square P_2) \geq \gamma_{1/2}(G).$
	\end{proof}
	
	\begin{proposition}
		If $\gamma_{1/2}(G) = 1,$ then 
		\[ \gamma_{1/2}(G \square P_m) \geq \gamma_{1/2}(P_m).		\]
	\end{proposition}
	
	\begin{proof}
		Let $|V(G)| = n.$ Since $\gamma_{1/2}(G) = 1,$ there exists some $v \in V(G)$ such that $|N[v]| \geq \left\lceil \frac{n}{2} \right\rceil.$ Now at best, $|N[v]| = n,$ in which case, $\max \{ |N[v]| : v \in V(G \square P_m) \} = n+2.$ Thus we have 
		\[   \gamma_{1/2}(G \square P_m) \geq \left \lceil \frac{mn}{n+2} \right \rceil  \geq \left \lceil \frac{m}{3} \right \rceil \geq \left \lceil \frac{m}{6} \right \rceil = \gamma_{1/2}(P_m). \]

	\end{proof}

	\begin{proposition}
		If $\gamma_{1/2}(G) = 2,$ then
		\[ \gamma_{1/2}(G \square P_m) \geq 2\gamma_{1/2}(P_m).		\]
	\end{proposition}
	
	\begin{proof}
		Let $|V(G)| = n.$ Since $\gamma_{1/2}(G) = 2,$ then $\max\{|N[v]| : v \in V(G)\} \leq \lceil \frac{n}{2} \rceil - 1.$ Therefore, $\max\{|N[v]| : v \in V(G \square P_m)\} \leq \lceil \frac{n}{2} \rceil + 1.$
		
		Thus we have 
		\[       \gamma_{1/2}(G \square P_m) \geq \left\lceil \frac{mn}{2\left( \lceil \frac{n}{2} \rceil + 1  \right)} \right\rceil \geq \left\lceil \frac{mn}{n+3} \right\rceil. \]
		
		Now, since $\gamma_{1/2}(G) = 2,$ we must have $n \geq 7.$ Thus $\gamma_{1/2}(G \square P_m) \geq \lceil 0.7m \rceil$.
		
		Now say $m=6a + b.$ Then 
		\[  \gamma_{1/2}(P_m) = \left\lceil \frac{m}{6} \right \rceil = \begin{cases} a & b = 0 \\ a+1 & b \neq 0  \end{cases}.  \]
		
		Thus for all $m \geq 2$,
		\[ \gamma_{1/2}(G \square P_m) \geq  \lceil 0.7m \rceil = \lceil 4.2a + 0.7b \rceil \geq 2(a+1) \geq 2\gamma_{1/2}(P_m). \] 
	\end{proof}

	\begin{proposition}
		If $\gamma_{1/2}(G) =  3,$ then 
		\[ \gamma_{1/2}(G \square P_m) \geq 3 \gamma_{1/2}(P_m).		\]
	\end{proposition}
	
	\begin{proof}
		Let $|V(G)|=n.$ Since $\gamma_{1/2}(G) = 3,$ we must have $n \geq 13.$ Moreover, we have $\max \{ |N[v]|: v \in V(G) \} \leq \left \lceil\frac{n}{2} \right\rceil -3.$ Thus, $\max \{ |N[v]|: v \in V(G\square P_m) \} \leq \left \lceil\frac{n}{2} \right\rceil -1 .$
		
		Therefore we have 
		\[  \gamma_{1/2}(G \square P_m) \geq \left\lceil \frac{mn}{2\left( \left\lceil \frac{n}{2} \right\rceil -1 \right) } \right\rceil \geq  \left\lceil \frac{mn}{n-1} \right\rceil  \geq m .\]
		
		Now say $m=6a + b.$ Then 
		\[  \gamma_{1/2}(P_m) = \left\lceil \frac{m}{6} \right \rceil = \begin{cases} a & b = 0 \\ a+1 & b \neq 0  \end{cases} . \]
		
		Thus for $m \geq 3,$
		\[  \gamma_{1/2}(G \square P_m) \geq m = 6a + b \geq 3(a + 1) \geq 3\gamma_{1/2}(P_m). \]
		
		Lastly, when $m = 2,$ we have 
		\[ \gamma_{1/2}(G \square P_2) \geq \left\lceil \frac{n}{ \left\lceil \frac{n}{2} \right\rceil -1 } \right\rceil = 3 = 3 \gamma_{1/2}(P_2). \] 
		
	\end{proof}
	
	%
	%
	%
	%

	%

	
	\section{Locating Partial Dominating Sets in Graphs} \label{Sec:locating}
	When we look a graph we want some tools (theorems and/or algorithms) that will help us locate a partial dominating set or elements of a partial dominating set. A $\gamma_p$ set does not have to be unique; and for many applications having anyone of them would work.  The first intuitive idea in looking for a partial dominating set is to consider vertices with high degrees. The following several results will explore what can and cannot be achieved following this idea. 
	
	\begin{lemma}
		For any $p\in [0,1]$, if a vertex $v \in G$ has the highest degree, then there is a $\gamma_p$-set that contains at least one of the following:
		\begin{itemize}
			\item[1)] $v$,
			\item[2)] a neighbor of $v$, i.e. an element of $N(v)$,
			\item[3)] a distance two neighbor of $v$, i.e. an element of $N(N[v])$.
		\end{itemize} 
	\end{lemma}
	\begin{proof}
		Suppose none of these vertices is in any $\gamma_p$ set. Let $v\in G$ be a vertex of highest degree.  If $S\subseteq G$ is a $p$-dominating set of $G$, consider any $s\in S$. The degree of $s$ is less than or equal to the degree of $v$. So the set 
		\[S'= \{v\} \cup (S \setminus \{s\}) \]
		dominates at least as many vertices as $S$. Thus, $S'$ is a $p$-dominating set that contains $v$. 
	\end{proof}
	
	Now we illustrate with some examples that in the preceding theorem, it may be the case that 2) or 3) hold and not 1).  
	
	\begin{example}
		Let $p = 8/9$ consider the graph in Figure \ref{fig:case3}. The set of boxed vertices is the only $p-$dominating set.  This shows that the highest degree vertex may not be in any of the $p$-dominating sets of the graph but that instead some of its distance 2 neighbors are.  
		
		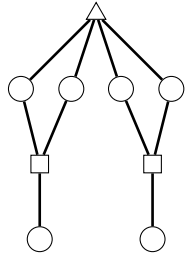
\begin{figure}[hh]
			\centering
			\begin{tikzpicture}
			\node [draw,regular polygon, regular polygon sides=3,inner sep=1.5pt] (A) at (0,1.5) {};
			\node [draw,circle] (B) at (-1,.5) {};
			\node [draw,circle] (C) at (-.33,.5) {};
			\node [draw,circle] (D) at (.33,.5) {};
			\node [draw,circle] (E) at (1,.5) {};
			\node [draw,rectangle] (F) at (-.75,-.5) {};
			\node [draw,rectangle] (G) at (.75,-.5) {};
			\node [draw,circle] (H) at (-.75,-1.5) {};
			\node [draw,circle] (I) at (.75,-1.5) {};
			
			\draw[line width=1pt] (A) edge (B);
			\draw[line width=1pt] (A) edge (C);
			\draw[line width=1pt] (A) edge (D);
			\draw[line width=1pt] (A) edge (E);
			\draw[line width=1pt] (B) edge (F);
			\draw[line width=1pt] (C) edge (F);
			\draw[line width=1pt] (D) edge (G);
			\draw[line width=1pt] (E) edge (G);
			\draw[line width=1pt] (F) edge (H);
			\draw[line width=1pt] (G) edge (I);
			\end{tikzpicture}
			\caption{The highest degree vertex $v$ (triangle) is not in any $\gamma_{8/9}$-set; rather the distance 2 neighbors of $v$ (rectangles) form the only $\gamma_{8/9}$-set. \label{fig:case3}}
		\end{figure}
	\end{example}
	\begin{example}
		Let $p = 7/9$ and consider the graph in Figure \ref{fig:case2}. The triangle vertex, $v$, is a maximum degree vertex in the graph. Any two vertices chosen from $N(v)$ form a $\gamma_{7/9}$-set. Thus it may be that a highest degree vertex is not in any $p$-dominating set, but that instead some its neighbors are.
		
		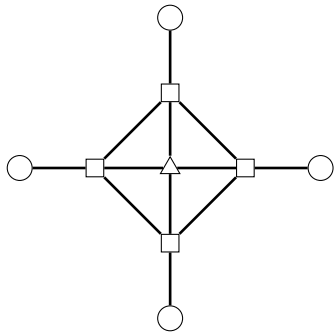
\begin{figure}[hhh]
			\centering
			\begin{tikzpicture}
			\node [draw,regular polygon, regular polygon sides=3,inner sep=1.5pt] (A) at (0,0) {};
			\node [draw,rectangle] (B) at (1,0) {};
			\node [draw,circle] (C) at (2,0) {};
			\node [draw,rectangle] (D) at (0,1) {};
			\node [draw,circle] (E) at (0,2) {};
			\node [draw,rectangle] (F) at (-1,0) {};
			\node [draw,circle] (G) at (-2,0) {};
			\node [draw,rectangle] (H) at (0,-1) {};
			\node [draw,circle] (I) at (0,-2) {};
			
			\draw[line width=1pt] (A) edge (B);
			\draw[line width=1pt] (B) edge (C);
			\draw[line width=1pt] (A) edge (D);
			\draw[line width=1pt] (D) edge (E);
			\draw[line width=1pt] (A) edge (F);
			\draw[line width=1pt] (F) edge (G);
			\draw[line width=1pt] (A) edge (H);
			\draw[line width=1pt] (H) edge (I);
			\draw[line width=1pt] (B) edge (D);
			\draw[line width=1pt] (D) edge (F);
			\draw[line width=1pt] (F) edge (H);
			\draw[line width=1pt] (H) edge (B);
			\end{tikzpicture}
			\caption{A highest degree vertex $v$ (triangle) is not in any $\gamma_{7/9}$-set; rather from the neighbors of $v$, any two together form a $\gamma_{7/9}$-set. \label{fig:case2}}
		\end{figure}
	\end{example}
	
	
	
	\begin{lemma}
		Let $p\in [0,1]$ be fixed. If $v\in V$ has highest degree and $v$ is not in any $p$-dominating set, then $|S\cap N[N[v]]|\geq 2$ for every $\gamma_p$-set $S.$ 
	\end{lemma}
	\begin{proof}
		Prove by contradiction; suppose $ |S\cap N[N[v]]| < 2$ for some $\gamma_p$-set $S.$ 
		In the case $|S\cap N[N[v]]|=0$, we can swap any $s\in S$ with $v$ to make a $p$-dominating set since $v$ has highest degree and none of its neighbors would be dominated. This contradicts $v$ not being in any $p$-dominating set. 
		
		In the case $|S\cap N[N[v]]|=1$, suppose $\{u\} = S \cap N[N[v]]$. 
		In the case $u \in N(v)$, we can swap $u$ and $v$ and still have a $p$-dominating set. 
		In the case $u \in N(N(v))$, we can again swap $u$ and $v$ and still have a $p$-dominating set.  In either case this again contradicts $v$ not being in any $p$-dominating set. 
	\end{proof}
	
	These results and examples together show that being greedy for the highest degree vertices does not work by itself in finding you elements from a $p$-dominating set, but that this greedy mindset can get you looking in the right area of the graph.  Observe that neither of graphs in Figures \ref{fig:case3} or \ref{fig:case2} were trees.  The graph if Figure \ref{fig:treecounterex} also shows that in a tree, a highest degree vertex need not be in a $p$-dominating set.
	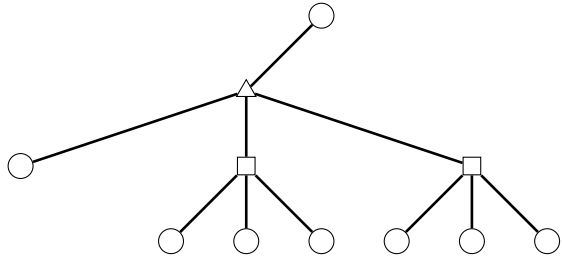
\begin{figure}[hhh]
		\centering
		\begin{tikzpicture}
		\node [draw,regular polygon, regular polygon sides=3,inner sep=1.5pt] (A) at (0,0) {};
		\node [draw,circle] (B) at (-3,-1) {};
		\node [draw,rectangle] (C) at (0,-1) {};
		\node [draw,rectangle] (D) at (3,-1) {};
		\node [draw,circle] (E) at (-1,-2) {};
		\node [draw,circle] (F) at (0,-2) {};
		\node [draw,circle] (G) at (1,-2) {};
		\node [draw,circle] (H) at (2,-2) {};
		\node [draw,circle] (I) at (3,-2) {};
		\node [draw,circle] (J) at (4,-2) {};
		\node [draw,circle] (K) at (1,1) {};
		
		\draw[line width=1pt] (A) edge (B);
		\draw[line width=1pt] (A) edge (C);
		\draw[line width=1pt] (A) edge (D);
		\draw[line width=1pt] (A) edge (K);
		\draw[line width=1pt] (C) edge (E);
		\draw[line width=1pt] (C) edge (F);
		\draw[line width=1pt] (C) edge (G);
		\draw[line width=1pt] (D) edge (H);
		\draw[line width=1pt] (D) edge (I);
		\draw[line width=1pt] (D) edge (J);
		\end{tikzpicture}
		\caption{A highest degree vertex $v$ (triangle) is not in any $\gamma_{9/11}$-set; rather two of its neighbors (rectangles)  form a $\gamma_{9/11}$-set. \label{fig:treecounterex}}
	\end{figure}

	\section{$p$-Influencing Set} \label{Sec:pinfluencing}
	Now we introduce a related definition by considering the union of all $p$-dominating sets of a graph. We will call this union the influencing set. 
	
	\begin{definition}
		Let $p\in [0,1]$ be fixed.  The union of all $\gamma_p$-sets of $G$ is called the \emph{$p-$influencing set} of $G$. 
	\end{definition}
	Note that there are $n = |V|$ interesting proportions $p$ that can be considered for a graph 
	\[  {1/n, 2/n,3/n,...,1}.\]
	Also we've allowed $p=0$, but here the $p$-dominating set is just $\emptyset$.
	For the smallest of these above $p=1/n$ the $p$-influencing set is all of $V$. 
	\begin{lemma}
		Let $p=1/n$, then the $p$-influencing set of $G$ is all of $V$.  If $G$ is a connected graph and $p = 2/n$, then the $p$-influencing set of $G$ is all of $V$.
	\end{lemma}
	\begin{proof}
		For the first, every individual vertex in $V$ is a $1/n$-dominating set. Thus the union of all $1/n$-dominating sets contains every vertex in $V$. 
		
		For the second, every individual vertex in $V$ is a $2/n$-dominating set, since it dominates itself and at least one neighbor. Thus the union of all $2/n$-dominating sets is all of $V$. 
	\end{proof}
	This lemma can be generalized further as follows.
	\begin{lemma}
		Let $\delta(G)$ be the smallest degree of any vertex in $G$. If $p \leq  \frac{\delta(G)+1}{n}$, then the $p$-influencing set of $G$ is all of $V$. 
	\end{lemma}
	\begin{proof}
		Every vertex is a $ \frac{\delta(G)+1}{n}$ dominating set, thus the $\frac{\delta(G)+1}{n}$-influencing set is all of $G$.
	\end{proof}
	As we see above, for the smaller interesting proportions the $p$-influencing set is as large as possible. One might ask if the other $p$-influencing sets have any containment properties as the proportion increases or decreases, or if the size of $p$-influencing set only decreases as $p$ increases. In general, this does not happen. Considering the graph in Figure \ref{fig:case3}, one can see that as $p$ runs from $1/9$ to 1 the $p-$influencing sets change from all the vertices down to one vertex and then back to all the vertices. Also in this example the intersection of all the $p$-influencing sets with $p>0$ is $\emptyset$.
	
	\begin{lemma}
		Let $\Delta(G)$ be the highest degree of any vertex in $G$. If $p=\frac{\Delta(G)+1}{n}$, then the $p$-influencing set is made up of exactly those vertices in $G$ with degree equal to $\Delta(G)$.
	\end{lemma}
	\begin{proof}
		When $p=\frac{\Delta(G)+1}{n}$, each vertex of degree $\Delta(G)$ is a $p$-dominating set. Furthermore, there are no $p$-dominating sets with more than one vertex, since they would not be of minimum size. Thus, the $p$-influencing set consists exactly of the degree $\Delta(G)$ vertices.  
	\end{proof}
	
	The previous results gave the $p$-influencing sets of any graph for a fixed $p.$ We now find all $p$-influencing sets for two common graphs, namely complete bipartite graphs and paths.
	
	\begin{proposition} Consider the complete bipartite graph $K_{m,n} = (V_1, V_2, E)$ with $|V_1| = m, |V_2| = n,$ and $m > n.$ The $p$-influencing sets of $K_{m,n}$ are
		\begin{itemize}
			\item $V_1\cup V_2$ for $0 < p \leq \frac{n+1}{m+n}$ and $p \geq \frac{m+2}{m+n}$
			\item $V_2$ for $\frac{n+2}{m+n} \leq p \leq \frac{m+1}{m+n}$
		\end{itemize}
		Moreover, if $m=n,$ then the $p-$influencing set of $K_{m,n}$ is $V_1 \cup V_2$ for all $p > 0.$
	\end{proposition}
	
	\begin{proof}
		Suppose $m > n.$ If $0 < p \leq \frac{n+1}{m+n},$ then any vertex of $K_{m,n}$ is a $\gamma_{p}$-set. Thus the $p$-influencing set is $V_1 \cup V_2.$ If $p \geq \frac{m+2}{m+n},$ then  $\{v_1,v_2\}$ where $v_1 \in V_1$ and $v_2 \in V_2,$ is a $\gamma_{p}$-set. Thus the $p$-influencing set is $V_1 \cup V_2.$ Lastly, if $\frac{n+2}{m+n} \leq p \leq \frac{m+1}{m+n},$ then any single vertex from $V_2$ is a $\gamma_{p}$-set, and no vertex from $V_1$ can be in a $\gamma_p$-set.  
		
		The argument is similar when $m=n.$
	\end{proof}

	\begin{corollary}
		Consider the complete bipartite graph $K_{m,n} = (V_1, V_2, E)$ with $|V_1| = m, |V_2| = n,$ and $m > n.$ The intersection of all the $p$-influencing sets ($p>0$) of $K_{m,n}$ is $V_2.$
		Moreover, if $m=n,$ the intersection is $V_1 \cup V_2.$
	\end{corollary}
	
	For the following proposition and corollary, we consider a path $P_n$ where the vertices are labeled $v_{i}$, $i \in \{1,2,\dots,n\}$ with $v_1, v_n$ as leaves and $N(v_i) = \{v_{i-1},v_{i+1}\}$ for $i \in \{2,3,\dots, n-1\}.$
	
	\begin{proposition}
		Consider a path $P_n$ with the vertices labeled as described above. The $p$-influencing sets of $P_n$ are given below. Note $k$ is a nonnegative integer chosen so that $0 < p \leq 1$ unless otherwise stated. 
		\begin{enumerate}
			\item For $n \equiv 0 \mod{3}$
			\begin{itemize}
				\item $V(P_n)$ if $p = \frac{3k+1}{n}$ or $p=\frac{3k+2}{n}$
				\item $\{v_2,v_3,\dots,v_{n-1}\}$ if $p=\frac{3k}{n} < 1$
				\item $\{v_2,v_5,v_8,\dots,v_{n-1}\}$ if $p=1$
			\end{itemize}
			\item For $n \equiv 1 \mod{3}$
			\begin{itemize}
				\item $V(P_n)$ if $p = \frac{3k+1}{n}$ or $p=\frac{3k+2}{n}$
				\item $\{v_2,v_3,\dots,v_{n-1}\}$ if $p=\frac{3k}{n}$ and $3k \neq n-1$
				\item $V(P_n)\setminus\{v_1,v_4,v_7,\dots,v_n\}$ if $p=\frac{n-1}{n}$
			\end{itemize}
			\item For $n \equiv 2 \mod{3}$
			\begin{itemize}
				\item $V(P_n)$ if $p = \frac{3k+1}{n}$ or $p=\frac{3k+2}{n} < 1$
				\item $\{v_2,v_3,\dots,v_{n-1}\}$ if $p=\frac{3k}{n}$ 
				\item $V(P_n)\setminus\{v_3,v_6,v_9,\dots,v_{n-2}\}$ if $p=1$
			\end{itemize}
		\end{enumerate}
		
	\end{proposition}
	
	\begin{proof}
		Suppose $n \equiv 0\mod{3}.$ Partition $V(P_n)$ into three sets \[ V_1 = \{v_1,v_4,\dots,v_{n-2}\}, V_2 = \{v_2,v_5,\dots,v_{n-1}\}, V_3 = \{v_3,v_6,\dots,v_n\}.\] 
		Now consider $p = \frac{3k+1}{n}$ or $p=\frac{3k+2}{n}.$ Then $\gamma_{p}(P_n) = k+1,$ and any $k+1$ vertices of $V_1, V_2,$ or $V_3$ comprise a $\gamma_{p}$-set. 
		
		Now consider $p = \frac{3k}{n} < 1.$ Then $\gamma_{p}(P_n) = k,$ and any $k$ vertices of $V_1\setminus\{v_1\}, V_2$ or $V_3\setminus\{v_n\}$ comprises a $\gamma_{p}$-set.
		
		Lastly, if $p=1,$ then $\gamma_1(P_n) = \frac{n}{3},$ and $V_2$ is the only $\gamma_{1}$-st. \\
		
		Suppose $n \equiv 1 \mod{3}.$ Partition $V(P_n)$ into three sets
		\[ V_1 = \{v_1,v_4,\dots,v_{n}\}, V_2 = \{v_2,v_5,\dots,v_{n-2}\}, V_3 = \{v_3,v_6,\dots,v_{n-1}\}.\]
		Note that $|V_1| = \frac{n-1}{3}+1,$ $|V_2|=|V_3| = \frac{n-1}{3}.$
		
		Now consider $p = \frac{3k+1}{n}$ or $p = \frac{3k+2}{n},$ then $\gamma_{p}(P_n) = k+1.$ Now if $p < 1,$ so $k < \frac{n-1}{3},$ then any $k+1$ vertices of $V_1, V_2,$ or $V_3$ comprise a $\gamma_{p}$-set. If $p=1,$ so $k=\frac{n-1}{3},$ then $V_1,$ $V_2 \cup\{v_n\},$ and $V_3\cup\{v_1\}$ are $\gamma_{1}(P_n)$-sets.
		
		Now if $p = \frac{3k}{n}$ and $3k \neq n-1,$ so $k < \frac{n-1}{3}.$ Then $\gamma_{p}(P_n) = k,$ and any $k$ vertices from $V_1 \setminus \{v_1,v_n\}, V_2,$ or $V_3$ comprise a $\gamma_{p}$-set. 
		
		Lastly, if $p=\frac{n-1}{n},$ then  $\gamma_p(P_n) = \frac{n-1}{3}.$ Then note that $V_2$ and $V_3$ are the only $\gamma_p(P_n)$-sets. \\

		Suppose $n \equiv 2\mod{3}.$ Partition $V(P_n)$ into three sets
		\[ V_1 = \{v_1,v_4,\dots,v_{n-1}\}, V_2 = \{v_2,v_5,\dots,v_{n}\}, V_3 = \{v_3,v_6,\dots,v_{n-2}\}.\]
		Note that $|V_1| = |V_2| = \frac{n+1}{3},$ and $|V_3|=\frac{n-2}{3}.$
		
		Now consider $p=\frac{3k+1}{n}$ or $p=\frac{3k+2}{n} < 1,$ so $k < \frac{n-2}{3}.$ Then $\gamma_{p}(P_n) = k+1$ and any $k+1$ vertices from $V_1, V_2,$ or $V_3$ comprise a $\gamma_p$-set.
		
		Now if $p=\frac{3k}{n},$ so $k \leq \frac{n-2}{3},$ then $\gamma_{p}(P_n) = k.$ Any $k$ vertices from $V_1\setminus{v_1}, V_2\setminus{v_n},$ or $V_3$ will be a $\gamma_p$-set. 
		
		Lastly, if $p=1,$ then $\gamma_1(P_n) = \frac{n+1}{3}.$ Note that $V_1$ and $V_2$ are the only $\gamma_{1}(P_n)$-sets. 
	\end{proof}

	\begin{corollary}
		The intersection of all $p$-influencing sets ($p > 0$) for a path $P_n$ is 
		\begin{enumerate}
			\item $\{v_{2+3k} | 0 \leq k \leq \frac{n-3}{3}\}$ if $n \equiv 0\mod{3}$
			\item $\{v_{2+3k},v_{3+3k} | 0 \leq k \leq \frac{n-4}{3}\}$ if $n \equiv 1\mod{3}$
			\item $\{v_{1+3k},v_{2+3j} | 0 < k \leq \frac{n-2}{3}, 0 \leq j < \frac{n-2}{3}\}$ if $n \equiv 2\mod{3}$.
		\end{enumerate}
	\end{corollary}

\end{document}